\documentclass[11pt,reqno]{amsart}

\usepackage[utf8x]{inputenc}
\usepackage{fullpage}
\usepackage[all]{xy}
\usepackage{amssymb, amsthm}
\usepackage{fourier}

\newcommand{\rk}{\mathrm{rk}}
\newcommand{\scal}[1]{\langle #1 \rangle}
\newcommand{\fg}{\mathfrak{g}}
\newcommand{\fp}{\mathfrak{p}}
\newcommand{\Spec}{\mathop{Spec}}

  \def\C{\mathbb{C}}
\def\Z{\mathbb{Z}}

\def\G{{\mathbb{G}}}
\def\tG{{\widetilde{G}}}

\DeclareMathOperator{\car}{char}
\DeclareMathOperator{\Iso}{{\mathcal{I}\!\!\textit{so}}}

\newtheorem{proposition}{Proposition}[section]
\newtheorem{lemma}[proposition]{Lemma}

\theoremstyle{definition}

\theoremstyle{remark}

\theoremstyle{plain} 
\newtheorem*{assumption}{Working assumption}
\newtheorem*{Mtheorem}{Main Theorem}

\begin{document}

\title[Stability of the tangent bundle of $G/P$]{Stability of the tangent bundle of 
$G/P$ in positive characteristics}

\author[I. Biswas]{Indranil Biswas}

\address{School of Mathematics, Tata Institute of Fundamental
Research, Homi Bhabha Road, Bombay 400005, India}

\email{indranil@math.tifr.res.in}

\author[P.-E. Chaput]{Pierre-Emmanuel Chaput}

\address{Domaine Scientifique Victor Grignard, 239, Boulevard des
Aiguillettes, Universit{\'e} Henri Poincar{\'e} Nancy 1, B.P. 70239,
F-54506 Vandoeuvre-l{\`e}s-Nancy Cedex, France}

\email{pierre-emmanuel.chaput@univ-lorraine.fr}

\author[Ch. Mourougane]{Christophe Mourougane}

\address{D{\'e}partement de Math{\'e}matiques,
Campus de Beaulieu, b\^{a}t. 22-23, Universit\'{e} de Rennes 1,
263 avenue du G\'{e}n\'{e}ral Leclerc, CS 74205, 35042 RENNES C\'{e}dex,
France}

\email{christophe.mourougane@univ-rennes1.fr}



\date{}

\begin{abstract}
Let $G$ be an almost simple simply-connected affine algebraic group over an algebraically 
closed field $k$ of characteristic $p\, >\, 0$. If $G$ has type $B_n$, $C_n$ or 
$F_4$, we assume that $p\,>\,2$, and if $G$ has type $G_2$, we assume that $p\,>\,3$. Let
$P\, \subset\, G$ be a parabolic subgroup. We prove that the tangent bundle of $G/P$ is 
Frobenius stable with respect to the anticanonical polarization on $G/P$.
\end{abstract}

\maketitle

\section{Introduction}

Let $G$ be an almost simple simply-connected affine algebraic group over an algebraically
closed field~$k$, and let $P\, \subset\, G$ be a parabolic subgroup. If the characteristic
$\car(k)$ is zero, then it is known that the tangent bundle of $G/P$ is stable with
respect to the anticanonical polarization on $G/P$. Throughout, (semi)stability means
slope (semi)stability. In the complex case
it was proved long ago that this bundle admits a K\"{a}hler-Einstein
metric (see \cite{koszul} or \cite[Chapter 8]{besse}),
which implies polystability. Simplicity of this bundle was proved in
\cite{AB}, proving the stability; A. Boralevi proved stability of $T(G/P)$
when $G$ is of type ADE \cite[Theorem C]{Bo}. Our aim here is to address stability of
$T(G/P)$ in the case where $\car(k)$ is positive.

If $G$ is of type $B_n$, $C_n$ or $F_4$, we assume that $\car(k)\,>\,2$; if $G$ is of type $G_2$, we 
assume that $\car(k)\,>\,3$.

The main Theorem of this note says that under the above assumption, the tangent bundle of 
$G/P$ and all its iterated Frobenius pull-backs are stable with respect to the 
anticanonical polarization on $G/P$.

The method of proof of the main Theorem is as follows. We prove that the
stability of $T(G/P)$ is equivalent to certain statement on the quotient
$\text{Lie}(G)/\text{Lie}(P)$ considered as a $P$--module. The statement in question
is shown to be independent of the characteristic of $k$ (as long as the above assumptions
hold). Finally, the main Theorem follows from the fact that $T(G/P)$ is stable
if $\car(k)\, =\, 0$.

A natural question to ask is whether $T(G/P)$ remains stable with respect to
polarizations on $G/P$ other than the anticanonical one. A. Boralevi gave a negative
answer to this question. She constructed examples of $G/P$ and polarization on them with
respect to which $T(G/P)$ is not even semistable \cite[Theorem D]{Bo}.

\section{Tangent bundle of $G/P$}

Let $G$ be an almost simple simply-connected affine algebraic group defined over an 
algebraically closed field $k$. The Lie algebra of $G$ will be denoted by $\fg$. Let 
$P\, \subsetneq\, G$ be a parabolic subgroup. We start with a result which is valid in 
all characteristics.

\begin{proposition}
 \label{prop:tensor-product}
Let $M_1,M_2$ be two $G$-modules such that $H^0(G/P, \, T(G/P))\,=\, M_1 \otimes M_2$
as $G$-modules. Then either $M_1\,=\,k$ or $M_2\,=\,k$.
\end{proposition}

\begin{proof}
 Let $\theta$ be the highest root of $\fg$. We claim that $\theta$ is a maximal weight of
$H^0(G/P,\, T(G/P))$ in the sense that $\theta + \alpha$ is not a weight of $H^0(G/P,\,T(G/P))$
for any positive root $\alpha$. To prove this, first note that if $H^0(G/P,\,T(G/P))\,=\,\fg$,
then this is in fact the definition of the highest root. By \cite[Th\'eor\`eme 1]{demazure},
there are only three cases
where $H^0(G/P,\,T(G/P)) \,\neq\, \fg$:
\begin{enumerate}
\item $G\,=\, \text{Sp}(2n)$ of type $C_n$ with $H^0(G/P,\,T(G/P))\,=\, \mathfrak{sl}(2n)$,

\item $G\,=\,\text{SO}(n+2)$ of type $B_n$ with $H^0(G/P,\,T(G/P))\,=\, \mathfrak{so}(2n+2)$, and

\item $G\,=\,G_2$ with $H^0(G/P,\,T(G/P))\,=\, \mathfrak{so}(7)$.
\end{enumerate}
In these three cases, we have exceptional automorphisms that account for additional vector fields
and we have $H^0(G/P,\, T(G/P))\,=\,\fg \oplus V$, where $V$ has a
unique highest weight which is not higher than $\theta$. For example, if $G\,=\,
\text{Sp}_{2n}$, then $G/P\,=\,\text{SL}(2n)/P_{\text{SL}(2n)}$ is a projective space of dimension $2n-1$, so that
 $H^0(G/P,\, T(G/P))$ is $\mathfrak{sl}(2n)$. Then $V$ is a module with unique highest
weight $\epsilon_1+\epsilon_2$, whereas $\theta\,=\,2\epsilon_1$ (in the notation of
\cite[Chap VI, Planches]{bourbaki}). So the claim is proved.

As $\theta$ is a maximal weight of $H^0(G/P,\, T(G/P))\,=\, M_1 \otimes M_2$,
there are maximal weights $\omega_1$ and $\omega_2$ of $M_1$ and $M_2$ respectively,
such that
 \begin{equation}
 \label{equa:theta}
 \theta \,= \,\omega_1 + \omega_2\, .
 \end{equation}
Since $\omega_1$ and $\omega_2$ are maximal, they are dominant. In all types except $A_n$
and $C_n$, we have $\theta$ to be a fundamental weight. Therefore, from the equality in
(\ref{equa:theta}) it follows that either $\omega_1\,=\,0$ or $\omega_2\,=\,0$, hence the
proposition is proved in these cases.

For the remaining cases of $A_n$ and $C_n$, assume that $\omega_1 \,\neq\, 0$ and
$\omega_2 \,\neq\, 0$. Let $\varpi_i$ denote the $i$-th fundamental weight.
In case of $A_n$, we have $\theta\,=\,\varpi_1 + \varpi_n$, so
up to a permutation, $\omega_1\,=\,\varpi_1$ and $\omega_2\,=\,\varpi_n$.
Since the Weyl group orbits of both $\varpi_1$ and $\varpi_n$ have $n+1$ elements,
it follows that $\dim M_1 \geq n+1$ and $\dim M_2 \geq n+1$. This implies
that $\dim H^0(G/P,\,T(G/P)) \,\geq\, (n+1)^2$ which is a contradiction. In
case of $C_n$, we have $\theta \,= \,2 \varpi_1$, so similarly we get
$\omega_1\,=\,\omega_2\,=\,\varpi_1$, and $\dim H^0(G/P,\,T(G/P)) \, \geq \, (2n)^2$.
This is again a contradiction.
\end{proof}

\section{The main result}

We now impose the following assumptions on the characteristic of $k$:

\begin{assumption}
\mbox{}
\begin{itemize}
\item The characteristic $\car(k)$ of $k$ is positive, and 

\item $\car(k)$ is bigger than all the coefficients
 $\scal { \alpha^\vee\, , \beta } $ for all roots 
 $\alpha\, ,\beta$ of $G$ with $\alpha \,\neq\, \beta$.
\end{itemize}
\end{assumption}

In other words, if the root system of $G$ is simply-laced, then $\car(k)$ is only
assumed to be positive; if $G$ is any of $B_n$, $C_n$ and $F_4$, we assume that
$\car(k)\,>\,2$; if $G\,=\, G_2$, we assume that $\car(k)\,>\,3$.

\begin{Mtheorem}
Under the previous assumption, the tangent bundle $T(G/P)$ is Frobenius stable with respect
to the anticanonical polarization on $G/P$.
\end{Mtheorem}

We will divide the proof into several steps. The question of stability will be reduced to
characteristic zero. The reduction to characteristic zero is
achieved using the following construction: Let $G_\Z$ be the split simply-connected
Chevalley group scheme over $\Z$ having the same root system as $G$. By the theory of
reductive algebraic group schemes,
as the root system characterizes simply-connected groups up to isomorphism,
we have $G \,\simeq\, G_\Z \otimes \Spec k$. On the
other hand, we denote $G_\Z \otimes \Spec \C$ by
$G_\C$. There exists a parabolic group $P_\Z \,\subset\, G_\Z$ such that
$P_\Z \otimes \Spec k$ is conjugate to $P$. The parabolic subgroup
$P \otimes \Spec \C$ of $G_\C$ will be denoted by $P_\C$.

Fix a maximal torus $T\,\subset\, G$
and a Borel subgroup $B$. Assume $T \subset B \subset P$.
Let $R$ denote the set of roots of $\fg$.
The set of positive (respectively, negative) roots of $\fg$ will be denoted by
$R^+$ (respectively, $R^-$). The eigenspace corresponding to any
$\alpha \,\in\, R$ will be denoted by $\fg^\alpha$.

A subsheaf $E \,\subset\, T(G/P)$ is called \emph{$G$-stable}
if it is preserved by the left action of $G$ on $T(G/P)$. Since the left translation
action of $G$ on $G/P$ is transitive, any $G$-stable subsheaf of $T(G/P)$ is a subbundle.

The anticanonical line bundles of $G/P$ and $G_\C/P_\C$ are ample.
Fix the anticanonical polarization on $G/P$ and also on $G_\C/P_\C$.

\begin{proposition}
 \label{prop:p=0}
Let $E \,\subset\, T(G/P)$ be a $G$-stable subbundle of $T(G/P)$. There
exists a subbundle $E_\C \,\subset\, T(G_\C/P_\C)$ such
that $\rk(E_\C)\,=\,\rk(E)$ and $\deg(E_\C)\,=\,\deg(E)$.
\end{proposition}

\begin{proof}
Let $x_0\,=\,eP/P \,\in \,G/P$ be the base point. The set of roots $\alpha$
such that $\fg^\alpha \,\subset\, \fp$ will be denoted by $I(P)$. We have
$$
T_{x_0}(G/P) \,\simeq\, \fg / \fp \,\simeq \,\bigoplus_{\alpha \in R \setminus I(P)} \fg^\alpha\, .
$$

Sending a $G$-stable subbundle $V\, \subset\, T(G/P)$ to the $P$-module $V_{x_0}$
an equivalence between $G$-stable subbundles of $T(G/P)$ and
$P$-submodules of $T_{x_0}(G/P)$ is obtained. Let $M$ be the $P$-submodule of
$T_{x_0}(G/P)$ corresponding to $E$.
Since $M$ is a $T$-stable subspace of $\bigoplus_{\alpha \not \in I(P)} \fg^\alpha$,
there is a subset $I(M)\subset R \setminus I(P)$ such that $M \,= \,\bigoplus_{\alpha \in I(M)} \fg^\alpha$. 
By the following Lemma \ref{lemm:caracterisation}, we have
$$\forall~ \beta \,\in\, I(P)\, ,\ \ \forall ~\alpha \,\in\, I(M)\, , \ \
\alpha + \beta \,\in\, R \setminus I(P) \, \Longrightarrow \, \alpha + \beta \,\in\, I(M)\, .
$$
Thus, $M_\C \,:=\,\bigoplus_{\alpha \in I(M)} \fg^\alpha_\C$ is a $P_\C$-submodule of $T_{x_0}(G_\C/P_\C)$
and the subbundle $E_\C\,\subset\, T(G_\C/P_\C)$ 
corresponding to $M_\C $ satisfies the conditions in the proposition.
\end{proof}

In the following Lemma, we consider the vector space $\bigoplus_{\alpha\in R\setminus I(P)} \fg^\alpha$.
This is isomorphic as a vector space to $\fg/\fp$, and therefore has a natural $P$-module
structure.

\begin{lemma}
 \label{lemm:caracterisation}
 Let $I \,\subset\, R \setminus I(P)$ be a set of negative roots.
 Then the sum $M(I)\,:=\,\bigoplus_{\alpha \in I} \fg^\alpha$ is a $P$-stable
 submodule of $\bigoplus_{\alpha \in R\setminus I(P)} \fg^\alpha$ if, and
 only if,
\begin{equation}\label{e1}
\forall~ \beta \,\in\, I(P)\, ,\ \ \forall ~\alpha \,\in\, I\, , \ \
\alpha + \beta \,\in\, R \setminus I(P) \, \Longrightarrow \, \alpha + \beta \,\in\, I\, .
\end{equation}
\end{lemma}

\begin{proof}
Take $\alpha \,\in\, I$ and $\beta \,\in\, I(P)$ such that $\alpha + \beta \,\in\, R \setminus I(P)$. In
particular, we have $\beta \,\neq\, \pm \alpha$. Since $G$ is simply-connected, $\fg$ is the Lie algebra 
defined by Serre's relations (this is explained for example in \cite[Remark 
2.2.3]{cr}), so we can choose a basis of $\fg$ such that the coefficients of the Lie 
bracket are those of the Chevalley basis \cite{carter}. Consider the biggest integer
$p$ such that $\alpha - p \beta \,\in\, R$. This $p$ is smaller than the length of the
$\beta$-string of 
roots through $\alpha$ minus $1$ (since $\alpha + \beta \,\in\, R$), and thus, by 
the working Assumption, we have $p \,\leq\, \car(k)-2$.
This implies that $p+1\,<\,\car(k)$. It now 
follows from \cite[Theorem 4.2.1]{carter} that $[\fg^{\beta},\fg^\alpha] \,=\, 
\fg^{\alpha+\beta}$. Assuming that $M(I)$ is $P$-stable, we have it to be
$\fp$-stable, and therefore $\alpha + \beta \,\in\, I$.

On the other hand, let $U_{\beta} \,\subset\, G$ be the one-parameter additive subgroup 
corresponding to the root $\beta$. Since $U_{\beta} \cdot \fg^\alpha \,\subset\, 
\bigoplus_{k \geq 0} \fg^{\alpha + k \beta}$, from \eqref{e1} it follows
that $M(I)$ is $U_{\beta}$-stable for any root $\beta$ in $I(P)$, and 
thus $M(I)$ is $P$-stable.
\end{proof}

\begin{lemma}\label{lemm:polystable}
The tangent bundle $T(G/P)$ is polystable.
\end{lemma}

\begin{proof}
Let $E$ be the first term of the Harder-Narasimhan filtration of $T(G/P)$. First assume 
$E \,\neq\, T(G/P)$, so
\begin{equation}\label{e2}
\mu(E) \,> \, \mu(T(G/P))\, ,
\end{equation}
where $\mu$ denotes the ${\rm slope}$, namely the quotient of the degree by the rank.
Since the anticanonical polarization of $G/P$ is fixed by $G$, from
the uniqueness of the Harder-Narasimhan filtration it follows that
$E$ is $G$-stable. By Proposition \ref{prop:p=0} and 
stability of $T(G_\C/P_\C)$ in characteristic $0$ \cite[Theorem 2.1]{AB}, we thus have
$\mu(E) \,< \, \mu(T(G/P))$ which contradicts \eqref{e2}. So $T(G/P)$ is semistable.

We can then similarly argue with the polystable socle (cf. \cite[page 23, Lemma 
1.5.5]{HL}) of $T(G/P)$ to deduce that $T(G/P)$ is polystable.
\end{proof}

Since $T(G/P)$ is polystable there are non-isomorphic stable vector bundles
$E_1\, ,\ldots\, ,E_r$ of same slope such that the natural map
\begin{equation}\label{e6}
\bigoplus_{i=1}^r Hom(E_i,\, T(G/P)) \otimes E_i \ \longrightarrow T(G/P)
\end{equation}
is an isomorphism.
We note that $E_1\, ,\ldots\, ,E_r$ are unique up to permutations of
$\{1\, ,\ldots\, ,r\}$.

\begin{lemma}\label{lemm:Ej-equivariant}
 Take any $g \,\in\, G$ and integer $1\,\leq\, j \,\leq\, r$. Then
$g^* E_j \,\simeq\, E_j$ as vector bundles on $G/P$.
\end{lemma}

\begin{proof}
Let $\phi\, :\, G\times (G/P)\,\longrightarrow\, G/P$ be the left-translation
action. Let $p_2\, :\, G\times (G/P)\,\longrightarrow\, G/P$ be the projection
to the second factor. The action $\phi$ produces an isomorphism of vector bundles
\begin{equation}\label{e4}
\Phi\, :\, \bigoplus_{i=1}^r Hom(E_i,\, T(G/P)) \otimes \phi^* E_i
\,=\, \phi^*T(G/P) \,\longrightarrow\, p^*_2T(G/P)\,=\,
\bigoplus_{i=1}^r Hom(E_i,\, T(G/P))\otimes p_2^* E_i\, .
\end{equation}
For $i\, \not=\, \ell$, as $E_i$ and $E_l$ are stable of the same slope, we have
$$
Hom((\phi^* E_i)\vert_{\{e\}\times G/P}\, ,
(p^*_2 E_\ell)\vert_{\{e\}\times G/P})\,=\, Hom(E_i\, ,E_\ell)\,=\, 0\, .
$$
Hence, using semi-continuity,
\begin{equation}\label{e3}
Hom(\phi^* E_i\, , p^*_2 E_\ell)\,=\, 0\, .
\end{equation}
{}From \eqref{e3} it follows immediately that $\Phi$ in \eqref{e4} takes
$Hom(E_i,\, T(G/P)) \otimes \phi^* E_i$ to itself for every $1\, \leq\, i\, \leq\, r$.
In particular, we have
$Hom(E_j,\, T(G/P)) \otimes \phi^* E_j\,\simeq\,
Hom(E_j,\, T(G/P)) \otimes p_2^* E_j$. Fix $g \in G$: restricting to $\{g\} \times G/P$, we get
\begin{equation}\label{e5}
 Hom(E_j,\, T(G/P)) \otimes g^*E_j\,\simeq\,
Hom(E_j,\, T(G/P)) \otimes E_j.
\end{equation}
Since $E_j$ is
stable, we know that $g^*E_j$ is indecomposable. Now in view of the uniqueness of
the decomposition into a direct sum of indecomposable vector bundles (see \cite[p. 315,
Theorem 2]{At}), from \eqref{e5} we conclude that $g^* E_j \,\simeq\, E_j$.
\end{proof}

\begin{lemma}\label{lemm:equivariant}
For all $j\, \in\, [1\, ,r]$, the vector bundle $E_j$ is $G$-equivariant.
\end{lemma}

\begin{proof}
Fix an integer $1\, \leq\,j\, \leq\, r$.
We now introduce the group of symmetries of the vector bundle $E_j$:
Let $\tG$ denote the set of pairs $(g\, ,h)$, where $g \,\in\, G$ and $h
\,\in\, Aut(E_j)$, such that the diagram
 $$
 \xymatrix{
 E_j \ar[r]^{h} \ar[d] & E_j \ar[d] \\
 G/P \ar[r]_{g} & G/P
 }
 $$
 commutes. Since $E_j$ is simple, $Aut_{G/P}(E_j) \,\simeq\, \G_m$, and therefore we get a
central extension
$$ 1\,\longrightarrow\, \G_m \,\longrightarrow\, \tG \,\stackrel{pr_1}{\longrightarrow}
\, G \,\longrightarrow\, 1\, . $$
By Lemma \ref{lemm:Ej-equivariant}, the above homomorphism $pr_1$ is surjective. This
$\tG$ is an algebraic group. To see this, consider the direct image
$p_{2*}\Iso(\phi^*E_j\, , p^*_2E_j)$, where $\phi$ and $p_2$ are the
projections in the proof of Lemma \ref{lemm:Ej-equivariant}, and
$\Iso(\phi^*E_j\, , p^*_2E_j)$ is the sheaf of isomorphisms between the
two vector bundles $\phi^*E_j$ and $p^*_2E_j$. This direct image is a principal
${\mathbb G}_m$-bundle over $G/P$. The total space of this
principal ${\mathbb G}_m$-bundle is identified with $\tG$.

We consider the derived subgroup $[\tG\, ,\tG]$. Since $G$ is simple and not abelian, we have
$[G\, ,G]\,=\,G$, so $\pi([\tG\, ,\, \tG])\,=\,G$. The unipotent radical of $\tG$ is 
trivial. Indeed, the unipotent radical is mapped to the trivial subgroup of $G$ 
since $G$ is simple. Therefore it is included in $\G_m$ and so the unipotent 
radical is trivial. Since $\tG$ is reductive, $[\tG\, ,\tG]$ is semi-simple, hence 
a proper subgroup of $\tG$ (the radical of $\tG$ contains $\G_m$ hence $\tG$ is not semi-simple).
Thus the restriction of $pr_1$ to $[\tG\, ,\tG]$ is an 
isogeny. Since $G$ is simply-connected, the restriction of $pr_1$ to $[\tG\, ,\tG]$ is an isomorphism. 
Consequently, the tautological action of $[\tG\, ,\tG]$ on $E_j$ makes it a $G$-equivariant bundle.
\end{proof}

\begin{lemma}\label{lemm:r=1}
The integer $r$ in \eqref{e6} is $1$.
\end{lemma}

\begin{proof}
Since $Hom(E_1,T(G/P)) \otimes E_1$ is a direct summand of
$T(G/P)$ (see \eqref{e6}), from Lemma \ref{lemm:polystable} we know that
the slope of $Hom(E_1,T(G/P)) \otimes E_1$ coincides with the slope
of $T(G/P)$. In the proof of Lemma \ref{lemm:equivariant} we saw that
$Hom(E_1,T(G/P)) \otimes E_1$ is a $G$-equivariant direct summand of $T(G/P)$.
As $T(G_\C/P_\C)$ is stable, \cite[Theorem 2.1]{AB}, from Proposition 
\ref{prop:p=0} it now follows that $Hom(E_1,T(G/P)) \otimes E_1 \,=\, T(G/P)$.
\end{proof}

\begin{lemma}\label{lemm:f}
$\dim Hom(E_1\, ,T(G/P)) \,= \, 1$.
\end{lemma}

\begin{proof}
{}From Lemma \ref{lemm:r=1}
we have $H^0(G/P,\, T(G/P)) \,=\, Hom(E_1,\, T(G/P)) \otimes H^0(G/P,\,E_1)$.
Since $T(G/P)$ is globally generated, so is $E_1$ and thus $\dim H^0(G/P,\,E_1)\,>\, 1$.
Thus, as $E_1$ is $G$-equivariant,
the lemma follows from Proposition \ref{prop:tensor-product}.
\end{proof}

{}From Lemma \ref{lemm:polystable}, Lemma \ref{lemm:r=1} and Lemma \ref{lemm:f}
it follows that $T(G/P)$ is stable.

The following lemma completes the proof of the main Theorem.

\begin{lemma}\label{lemm:frobenius}
Let $E$ be a semi-stable (respectively, stable) $G$-equivariant vector bundle on
$G/P$. Then $E$ is Frobenius semi-stable (respectively, Frobenius stable). 
\end{lemma}

\begin{proof}
The absolute Frobenius morphism on $G/P$ will be denoted by $F$. First assume that $E$ is 
semi-stable. Let again $W$ be the first term of the Harder-Narasimhan filtration 
of $F^*E$. We use the correspondence between vector bundles on $G/P$ and 
$P$-modules. Thus $W$ corresponds to a $P$-stable subspace of $(F^*E)_{x_0}$, the fiber 
of $F^*E$ at the base point in $G/P$. This is the same as an $F^*P$-stable subspace $S$
of $E_{x_0}$. Since $F\,:\,P \,\longrightarrow\, P$ is bijective, this $S$ is also
a $P$-submodule of $E_{x_0}$. 
Thus, there exists a subbundle $W' \,\subset\, E$ 
of slope $\frac{\mu(W)}{p} \,\geq\, \frac{\mu(F^*E)}{p}\,=\,\mu(E)$
such that $W \,= \,F^* W'$. By 
semi-stability of $E$, we have $W'\,=\,E$. Thus we get that $W\,=\, F^*E$.
 
Assume now that $E$ is stable. So $F^*E$ is semistable.
Let $W \,\subset\, F^*E$ be a subbundle with $\mu(W)
\,=\, \mu(F^*E)$. We consider the Cartier connection $F^*E \,\longrightarrow\,
F^*E \otimes \Omega^1_{G/P}$. The subbundle $W$ is a Frobenius pull-back if and only
if its image 
under the composition $$W \,\longrightarrow\, F^*E
\,\longrightarrow\,F^*E \otimes \Omega^1_{G/P}$$ is contained in 
$W\otimes \Omega^1_{G/P}$. Since both $E$ and $\Omega^1_{G/P}$ are Frobenius
semistable, the tensor product $E \otimes \Omega^1_{G/P}$ is again 
semi-stable \cite[p. 285, Theorem 3.18]{RR}.
But $\mu(F^*E \otimes \Omega^1_{G/P}) < \mu(F^*E) \,=\, \mu(W)$, so 
this composition vanishes. Therefore, let $W' \subset E$ be such that $W \,= \,F^* W'$. 
We have $\mu(W') \,=\, \mu(E)$. By stability of $E$, we get that $W'\,=\,E$ and hence $W=F^\star E$.
\end{proof}

\section{An example in small characteristic}

We give an example of a tangent bundle which is semi-stable but not stable. We do not know
if there are some tangent bundles to homogeneous spaces which are not semi-stable.

The example is that of $X=G/P=\G_\omega(n,2n)$, the Grassmannian of Lagrangian spaces in
a symplectic space of dimension $2n$, and we assume that $k$ has characteristic $2$.
Namely, $G$ is $Sp_{2n}$ and $P$ corresponds to the long simple root.
Let $U$ denote the universal bundle on $X$, of rank $n$ and degree $-1$.
Then $TX$ is a subbundle of $U^* \otimes U^*$; in fact if $S^2U$ denotes
the symmetric quotient of $U \otimes U$, then $TX \simeq (S^2U)^*$.

We will implicitly use the correspondence between $P$-modules and $G$-linearized homogeneous bundles on $X$.
Note that the reductive quotient of $P$ is $GL(U)$. Since there is an injection $F^*U \to S^2U$ of $GL(U)$-modules
($F$ denotes the Frobenius morphism), this defines an exact sequence of bundles on $X$:
\begin{equation}
 0 \rightarrow F^*U \rightarrow S^2U \rightarrow K \rightarrow 0
\end{equation}
It follows that
there is a subbundle $K^* \subset TX$. Since $\mu(F^*U)=\mu(S^2U)=2\mu(U)$,
we get $\mu(K^*)=\mu(TX)$ and $TX$ is not
stable. However since $F^*U$ is the only
$GL(U)$-invariant subspace in $S^2U$, $K^*$ is the only equivariant subbundle
in $TX$. Thus the semi-stability
inequality holds for this subbundle. Arguing as in the proof of Lemma \ref{lemm:polystable},
we deduce that $TX$ is semi-stable.

\vskip 5mm

For general homogeneous spaces $G/P$, we face two difficulties:
\begin{itemize}
 \item There are equivariant subbundles in $TX$ which do not lift
to characteristic $0$, and contrary to the above example, they are numerous in general.

\item The stability of $TX$ for characteristic $0$ says nothing about $\mu(E)$ of such a 
subbundle $E \,\subset\, TX$. It is difficult to compute the $(\dim(G/P)-1)$-th power of 
the anticanonical polarization to be able to show the semi-stability inequality for $E$.
\end{itemize}

\section*{Acknowledgements}

We are very grateful to G. Ottaviani for pointing out an error in a previous version.
He also brought \cite{Bo} to our attention.
The second and third authors thank the Tata Institute of Fundamental Research, while
the first author thanks Institut de Math\'{e}matiques de Jussieu for hospitality
during various stages of this work.

\end{document}